\documentclass[a4paper]{article}

\usepackage[utf8]{inputenc}
\usepackage{lmodern}
\usepackage[OT1]{fontenc} %T1 has font licensing issues. Use QX if possible, or OT1.
\usepackage[english]{babel}
\usepackage{amsmath,amsthm,amssymb}
\usepackage{mathrsfs}
\usepackage{graphicx}
\usepackage{framed}
\usepackage{xcolor}
\usepackage[normalem]{ulem}

\theoremstyle{plain}
\newtheorem{theorem}{Theorem}[section]
\newtheorem*{theorem*}{Theorem}
\newtheorem{lemma}[theorem]{Lemma}

\newtheorem{proposition}[theorem]{Proposition}
\newtheorem*{proposition*}{Proposition}

\newtheorem*{conjecture*}{Conjecture}

\theoremstyle{definition}
\newtheorem{definition}[theorem]{Definition}

\theoremstyle{remark}
\newtheorem{remark}[theorem]{Remark}

\DeclareMathOperator{\supp}{supp}

\renewcommand{\d}{\partial}
\newcommand{\db}{\overline{\partial}}
\newcommand{\Ca}{{\db^{-1}}}
\newcommand{\Cab}{{\d^{-1}}}
\newcommand{\conjCa}{\mathcal{T}}

\newcommand{\abs}[1]{\left\lvert #1 \right\rvert}
\newcommand{\norm}[1]{\left\lVert #1 \right\rVert}

\newcommand{\R}{\mathbb{R}}
\newcommand{\C}{\mathbb{C}}

\newcommand{\N}{\mathbb{N}}

\renewcommand{\d}{\partial}

\newcommand{\SinfinfExp}{\alpha}

\title{Uniqueness for the inverse boundary value problem with singular potentials in 2D} \author{Eemeli Bl{\aa}sten\thanks{HKUST Jockey Club Institute for Advanced Study, Hong Kong.
Email: iaseemeli@ust.hk}\qquad Leo Tzou\thanks{Faculty of Science, University of Sydney, Sydney, Australia. Email: leo.tzou@sydney.edu.au}\qquad Jenn-Nan Wang\thanks{Institute of Applied Mathematical Sciences, NCTS, National Taiwan University, Taipei 106, Taiwan. Email: jnwang@math.ntu.edu.tw}}

\begin{document}
\maketitle
\begin{abstract}
In this paper we consider the inverse boundary value problem for the Schr\"odinger equation with potential in $L^p$ class, $p>4/3$. We show that the potential is uniquely determined by the boundary measurements.
\end{abstract}

\section{Introduction}

In this work we study the inverse boundary value problem for the Schr\"odinger equation with singular potentials in the plane. Let $\Omega\subset\R^2$ be an open bounded domain with Lipschitz boundary $\partial\Omega$. Let $q\in L^p(\Omega)$ with $p>1$ and assume that $0$ is not a Dirichlet eigenvalue of the Schr\"odinger operator $\Delta+q$ in $\Omega$. Then the Dirichlet-Neumann map $\Lambda_q:u|_{\partial\Omega}\mapsto\partial_{\nu}u|_{\partial\Omega}$, where $u$ satisfies $\Delta u+qu=0$ in $\Omega$, is well-defined (see \cite[Lemma~5.1.3]{blasten} for the precise statement). Here we are concerned with the unique determination of $q$ from the knowledge of $\Lambda_q$, namely, whether 
\begin{equation}\label{fq}
\Lambda_{q_1}=\Lambda_{q_2}\Rightarrow q_1=q_2.
\end{equation}
We will prove that \eqref{fq} is indeed true for $L^p(\Omega)$ potentials with $p>4/3$.   

Our method follows from the strategies introduced by Bukhgeim in \cite{bukhgeim} where he showed that \eqref{fq} holds true for $C^1$ potentials. The key ingredient in Bukhgeim's method is the invention of special complex geometrical optics solutions with non-degenerate singular phases, namely, $\Phi(z)=(z-z_0)^2$, $z,z_0\in\C$. Using this type of complex geometrical optics solutions, Bukhgeim was able to prove the global uniqueness by the method of stationary phase. Bukhgeim's result was later improved to $L^p(\Omega)$ with $p>2$ in \cite{imanuvilovYamamotoLp} and \cite{Blasten--Imanuvilov--Yamamoto}.  

In this paper, we push the uniqueness result even further to $p>4/3$. To do so, we need to prove the existence of complex geometrical optics solutions with phase $\Phi$ for such potentials. In fact, we show that such complex geometrical optics solutions exist for all potentials in $L^p(\Omega)$ with $p>1$. The improvement relies on a new estimate for the conjugated Cauchy operator (see Lemma~\ref{fundLpEstim}). Having constructed the complex geometrical optics solutions, we then perform the usual step --- substituting such special solutions into Alessandrini's identity. In order to obtain the dominating term containing the difference of potentials in the method of stationary phase, we need to derive more refined estimates of terms of various orders in Alessandrini's identity. In this step, we need to use the fact that the knowledge the DN map improves the integrability of the potential. In other words, $\Lambda_{q_1}=\Lambda_{q_2}$ for $q_1, q_2\in L^p(\Omega)$ with $3/4<p<2$ implies $q_1-q_2\in L^2(\Omega)$ (see \cite{Serov--Paivarinta}).  

We would also like to mention another related paper. It was shown in \cite[Thm 2.3]{lakshtanov2016recovery} that if $p>1$ then for every $z_0 \in \Omega$ there exists a \emph{generic set} of potentials in $L^p$ for which its value in a neighbourhood of $z_0$ is recoverable. It was also remarked in \cite{lakshtanov2016recovery} that the neighbourhood of $z_0$ also depends on the chosen potential in the generic family which is determined by the choice of $z_0\in \Omega$. Though the assumption on the $L^p$ space of the potential is more general, the dependence of the generic set on the choice of the point $z_0\in\Omega$ and the dependence of the neighbourhood on the potential makes it unclear how a global identifiability result would follow from  \cite[Thm 2.3]{lakshtanov2016recovery}.

Intuitively, the uniqueness of the inverse boundary value problem is strongly related to the unique continuation property.  For higher dimensions $(n\ge 3)$, it is known that the unique continuation holds for any solution $u\in H^{2,\frac{2n}{n+2}}_{loc}(\Omega)$ when $q\in L^{n/2}_{loc}$ (scale-invariant potentials) \cite{Jerison-Kenig-Unique}, where $H^{2,s}_{loc}(\Omega)=\{u\in L^1_{loc}(\Omega):\Delta u\in L^s_{loc}(\Omega)\}$. In this situation, the global uniqueness of the inverse boundary value problem with $q\in L^{n/2}$ was established in  \cite{DSFKS13} (also see related result in \cite{Haberman16} for $n=3, q\in W^{-1,3}$). When $n=2$, the unique continuation holds relative to $u\in H^{2,s}_{loc}$ for $q\in L^p_{loc}$ with any $p>1$, where $s=\max\{1,\frac{2p}{p+2}\}$ \cite{ABG81}. Prior to \cite{ABG81}, a weaker result which stated that the unique continuation property holds relative to $u\in H^{2,2}_{loc}$ for $q\in L^p_{loc}$ with $p>4/3$ was proved in \cite{Saut-Scheurer-1980}. Our uniqueness theorem of the inverse boundary value problem is consistent with the unique continuation result relative to $u\in H^{2,2}_{loc}$. It remains an interesting problem to close the gap of the uniqueness theorem for the inverse boundary value problem for $q\in L^p(\Omega)$ with $1<p\le 4/3$.

\section{Main results}
\begin{theorem} \label{uniquenessThm}
  Let $q_1,q_2\in L^p(\Omega)$ with $4/3<p<2$. Assume $0$ is not a
  Dirichlet-eigenvalue of either potential and that their
  Dirichlet-Neumann maps are identical $\Lambda_{q_1} =
  \Lambda_{q_2}$. Then $q_1=q_2$.
\end{theorem}

Let us fix some notational convention before stating our second
theorem. Throughout this text we shall always assume the following,
which we call \emph{the usual assumptions}: Let $\Omega,X\subset\R^2$
be bounded domains and $\overline\Omega \subset\subset X$. Fix a
cut-off function $\chi\in C^\infty_0(X)$ such that $\chi\equiv1$ on
$\Omega$. Also, whenever we let a function belong to $L^p(\Omega)$ for
any $p$ we automatically extend it by zero to $L^p(X)$. Finally, let
$\tau>1$ and $z_0\in\C \equiv \R^2$ and set $\Phi(z) =
(z-z_0)^2$. Moreover should $\varphi_j$ or $S_j$ be mentioned, then
they refer to definitions \ref{varphiDef} and \ref{Sdef}. We now state the existence of complex geometrical optics solutions.

\begin{theorem} \label{CGOexistence}
  Let the usual assumptions hold and $q_j \in L^p(\Omega)$ with
  $1<p<2$. For $j\in\{1,2\}$ let $\beta_j = \beta_j(z_0)$ be uniformly
  bounded over a parameter $z_0\in\C$. Then for any $z_0$ we may
  define the function $\varphi_j$ from Definition~\ref{varphiDef} and
  the series
  \[
  f_j(z) = \sum_{m=0}^\infty F_{j,m}(z) = e^{-i\tau(\Phi+\overline\Phi)} +
  \varphi_j(z) + S_j \varphi_j(z) + \ldots
  \]
  from Definition \ref{fDef}. The latter converges uniformly in the
  variable $z \in X$ for $\tau$ large enough and
  \[
  \norm{F_{j,m}}_\infty \leq (C \tau^{-\SinfinfExp})^m
  \]
  for some $C=C(p,\chi,q_j,\sup_{z_0}\abs{\beta_j})$, where
  $\SinfinfExp>0$ is a constant obtained in Proposition
  \ref{Sestimates}. Moreover we have $f_j \in W^{1,2}(X)$. Lastly,
  $(\Delta+q_1)(e^{i\tau\Phi}f_1) = 0$ and
  $(\Delta+q_2)(e^{i\tau\overline\Phi}f_2) = 0$ in $\Omega$.
\end{theorem}

\section{Notation and CGO solution buildup}
In this section we shall start by defining the operators and notation
used in the rest of the paper. At the same time we reduce the
construction of Bukhgeim-type \cite{bukhgeim} complex geometrical
optics solutions to an integral equation.

We shall use the complex geometrical optics solutions for
$(\Delta+q_j)u_j=0$ of the form
\[
u_1 = e^{i\tau\Phi} f_1, \qquad u_2 = e^{i\tau\overline\Phi} f_2.
\]
The special form of $f_1,f_2$ which we are going to use was first
defined in \cite{imanuvilovYamamotoLp} and used in
\cite{Blasten--Imanuvilov--Yamamoto} for proving uniqueness for the
boundary value inverse problem when the potentials are in $L^p$,
$p>2$.

\begin{definition} \label{dDef}
  With the usual assumptions, define the differential operators
  \begin{align*}
    D_1 f &= -4 e^{-i\tau(\Phi+\overline\Phi)} \d(
    e^{i\tau(\Phi+\overline\Phi)} \db f),\\ D_2 f &= -4
    e^{-i\tau(\Phi+\overline\Phi)} \db( e^{i\tau(\Phi+\overline\Phi)}
    \d f ).
  \end{align*}
\end{definition}

\begin{lemma} \label{conjPDE}
  Let the usual assumptions hold. For $q_j \in L^1(\Omega)$ and $f_j
  \in L^\infty(\Omega)$ we have
  \begin{align*}
    (\Delta+q_1)( e^{i\tau\Phi}f_1 ) &= 0 \qquad \Leftrightarrow
    \qquad D_1f_1 = q_1f_1,\\ (\Delta+q_2)( e^{i\tau\overline\Phi}f_2 )
    &= 0 \qquad \Leftrightarrow \qquad D_2f_2 = q_2f_2,
  \end{align*}
  all in $\Omega$.
\end{lemma}
\begin{proof}
  Use $\Delta = 4 \d\db = 4 \db\d$ for distributions on $\Omega$.
\end{proof}

For inverting the operators $D_1$ and $D_2$ we will have to use
conjugated versions of the \emph{Cauchy operators} $\Ca$ and
$\Cab$. We have included a short reminder of their properties in
Section \ref{cauchySection}.

\begin{definition} \label{Sdef}
  Let the usual assumptions hold. The we define the operators
  $S_1,S_2$ acting on $f\in L^\infty(X)$ by
  \begin{align*}
    S_1f &= -\frac{1}{4} \Ca\big( e^{-i\tau(\Phi+\overline\Phi)} \chi
    \Cab( e^{i\tau(\Phi+\overline\Phi)} q_1 f) \big), \\ S_2f &=
    -\frac{1}{4} \Cab\big( e^{-i\tau(\Phi+\overline\Phi)} \chi \Ca(
    e^{i\tau(\Phi+\overline\Phi)} q_2 f) \big).
  \end{align*}
\end{definition}

\begin{remark}
  We have $D_j S_j f = q_j f$ in $\Omega$ (but not in $X \setminus
  \Omega$).
\end{remark}

\begin{definition} \label{varphiDef}
  Let the usual assumptions hold and $q_j \in L^1(\Omega)$. Then for
  any given $z_0\in\C$ and $\beta_j = \beta_j(z_0)$, we define
  functions of $z\in X$ by
  \begin{align*}
    \varphi_1 &= \frac{1}{4} \Ca\big( e^{-i\tau(\Phi+\overline\Phi)}
    \chi (\beta_1(z_0) - \Cab q_1)\big), \\ \varphi_2 &= \frac{1}{4}
    \Cab\big( e^{-i\tau(\Phi+\overline\Phi)} \chi (\beta_2(z_0) - \Ca
    q_2)\big).
  \end{align*}
  Note that $\Cab q_1$ and $\Ca q_2$ inside the parenthesis do not
  depend on $z_0$. For example
  \[
  \varphi_1(z) = \frac{1}{4\pi} \int \frac{
    e^{-i\tau((z'-z_0)^2+(\overline{z}'-\overline{z_0})^2)} \chi(z')
    \left( \beta_1(z_0) - \Cab q_1(z')\right) }{z-z'} dm(z').
  \]
\end{definition}

\begin{definition} \label{fDef}
  Let the usual assumptions hold. For $j\in\{1,2\}$ define
  \[
  f_j = e^{-i\tau(\Phi+\overline\Phi)} + \sum_{m=0}^\infty S_j^m
  \varphi_j
  \]
  where $S_j$ is as in Definition \ref{Sdef} and $\varphi_j$ as in
  Definition \ref{varphiDef}. For convenience we write
  \[
  F_{j,0} = e^{-i\tau(\Phi+\overline\Phi)}, \qquad F_{j,m} = S_j^{m-1}
  \varphi_j
  \]
  when $m\in\N$, $m\geq1$. Hence $f_j = \sum_{m=0}^\infty F_{j,m}$.
\end{definition}

We have made enough definitions now to show the structure of the
complex geometrical optics solutions. Given $z_0\in\C$ and $\beta_j =
\beta_j(z_0)$ we can show formally that if
\[
u_1 = e^{i\tau\Phi} f_1, \qquad u_2 = e^{i\tau\overline\Phi} f_2,
\]
then $(\Delta+q_j)u_j=0$ in $\Omega$. This follows from writing
\[
f_j = e^{-i\tau(\Phi+\overline\Phi)} + \varphi_j + S_j(f_j -
e^{-i\tau(\Phi+\overline\Phi)}),
\]
applying $D_j$ and noting that
$D_j(e^{-i\tau(\Phi+\overline\Phi)})=0$, $D_j S_j$ is the
multiplication operator by $q_j$ in $\Omega$, and $D_j \varphi_j = q_j
e^{-i\tau(\Phi+\overline\Phi)}$ in $\Omega$. Then $D_j f_j = q_j f_j$
and Lemma \ref{conjPDE} gives the rest. For proving actual existence
and estimates, see the proof at the end of the next section.

\section{Estimates for conjugated operators and CGO existence}

In this section we will start by showing that a fundamental operator
of the form $a \mapsto \Ca( e^{-i\tau(\Phi+\overline\Phi)} a )$ has
decay properties as $\tau\to\infty$. Section \ref{cutoffs} contains
the technical cut-off function estimates. Once estimates for this
fundamental operator have been shown we can prove the required
estimates for $S_j$ from Definition \ref{Sdef}. At the end of this
section all the details for proving the existence of complex
geometrical optics solutions for $L^p$-potentials with $1<p<2$, i.e.
Theorem \ref{CGOexistence}, will be given. From now on, we define
${p*}$ satisfying the relation
\[
\frac{1}{p}=\frac{1}{2}+\frac{1}{p*}.
\]
Thus, we have $p*>2$.

\begin{definition} \label{Tdef}
  Let the usual assumptions hold. Define the operator $\conjCa$ by
  \[
  \conjCa a = \Ca( e^{-i\tau(\Phi + \overline\Phi)} a ).
  \]
\end{definition}

\begin{lemma} \label{fundLinftyEstim}
  Let the usual assumptions hold and $\conjCa$ as in Definition
  \ref{Tdef}. Then, for ${p*}>2$ we can extend $\conjCa$ to a mapping
  $W^{1,{p*}}_0(X) \to L^\infty(X)$ with norm estimate
  \[
  \norm{\conjCa a}_{L^\infty(X)} \leq C \tau^{-1/{p*}}
  \norm{a}_{W^{1,{p*}}(X)},
  \]
  where $W^{1,{p*}}_0(X)$ is the completion of $C^\infty_0(X)$ under
  the $W^{1,{p*}}$-norm.
\end{lemma}
  
\begin{proof}
  Let $\psi \in C^\infty_0(\R^2)$ be a test function supported in
  $B(\bar0,2)$ with $0\leq\psi\leq1$ and $\psi \equiv 1$ in
  $B(\bar0,1)$. Write $\psi_\tau(z) = \psi( \tau^{1/2}(z-z_0) )$. Let
  $h(z) = (1 - \psi_\tau(z)) / (\overline{z} - \overline{z_0})$. By
  integration by parts (Lemma \ref{intByPartsLemma}) we have
  \begin{align*}
    &\Ca(e^{-i\tau(\Phi + \overline{\Phi})}a) = \Ca\big(e^{-i\tau(\Phi
      + \overline{\Phi})} \psi_\tau a\big) \\ &\qquad\quad -
    \frac{1}{2i\tau} \big( e^{-i\tau(\Phi + \overline{\Phi})} h a -
    \Ca(e^{-i\tau(\Phi + \overline{\Phi})} \db h a) - \Ca(
    e^{-i\tau(\Phi + \overline{\Phi})} h \db a) \big).
  \end{align*}
  Then recall that by Lemma \ref{cauchyNorm} we have $\Ca \colon
  L^{p*}(X) \to W^{1,{p*}}(X)$, the latter of which is embedded into
  $L^\infty(X)$ since ${p*}>2$. Hence, by taking the
  $L^\infty(X)$-norm we have
  \[
  \norm{ \conjCa a}_\infty \leq C\big( \norm{ \psi_\tau a}_{p*} +
  \tau^{-1}( \norm{ h a }_\infty + \norm{ \db h a }_{p*} + \norm{ h
    \db a }_{p*})\big).
  \]
  The claim follows from H\"older's inequality and lemmas
  \ref{bumpNorm} and \ref{hNorm} after estimating
  \[
  \norm{ \conjCa a}_\infty \leq C \big( \norm{\psi_\tau}_{p*} +
  \tau^{-1}( \norm{h}_\infty + \norm{ \db h}_{p*}) \big) (
  \norm{a}_\infty + \norm{\db a}_{p*})
  \]
  and noting that $\tau^{-1/2} \leq \tau^{-1/{p*}}$ since $\tau>1$.
\end{proof}

\begin{lemma} \label{fundLpEstim}
  Let the usual assumptions hold and $\conjCa$ be as in Definition
  \ref{Tdef}. Assume that $2<{p*}<\infty$ and $1/2 + 1/{p*} \geq 1/q >
  1/2$. Then we can extend $\conjCa$ to a mapping $W^{1,q}_0(X) \to
  L^{p*}(X)$ with norm
  \[
  \norm{\conjCa a}_{L^{p*}(X)} \leq C \tau^{1/q - 1 - 1/{p*}}
  \norm{a}_{W^{1,q}(X)}
  \]
  where $C = C(p,q,X)$.
\end{lemma}

\begin{proof}
  Let $\psi \in C^\infty_0(\R^2)$ be a test function supported in
  $B(0,2)$ with $0\leq\psi\leq1$ and $\psi \equiv 1$ in $B(0,1)$. For
  $\tau>0$ and $z_0\in\R^2$ write $\psi_\tau(z) = \psi(
  \tau^{1/2}(z-z_0) )$. Let $h(z) = (1 - \psi_\tau(z)) / (\overline{z}
  - \overline{z_0})$. Integration by parts gives
  \begin{align*}
    &\Ca(e^{-i\tau(\Phi + \overline{\Phi})}a) = \Ca\big(e^{-i\tau(\Phi
      + \overline{\Phi})} \psi_\tau a\big) \\ &\qquad\quad -
    \frac{1}{2i\tau} \big( e^{-i\tau(\Phi + \overline{\Phi})} h a -
    \Ca(e^{-i\tau(\Phi + \overline{\Phi})} \db h a) - \Ca(
    e^{-i\tau(\Phi + \overline{\Phi})} h \db a) \big)
  \end{align*}
  by Lemma \ref{intByPartsLemma}.

  Sobolev embedding and Lemma \ref{cauchyNorm} imply that $\Ca :
  L^p(X) \to L^{p*}(X)$. Taking the $L^{p*}(X)$-norm gives
  \[
  \norm{ \conjCa a}_{p*} \leq C \big( \norm{ \psi_\tau a}_p +
  \tau^{-1}( \norm{ h a }_{p*} + \norm{ \db h a }_p + \norm{ h \db a
  }_p) \big).
  \]
  Again, recall that $W^{1,q}(X) \hookrightarrow L^{q*}(X)$ where
  $1/{q*} = 1/q - 1/2$. H\"older's inquality gives
  \begin{align*}
    &\norm{ \psi_\tau a}_p \leq \norm{\psi_\tau}_{r_1} \norm{a}_{q*},
    && \frac{1}{r_1} = \frac{1}{p} - \frac{1}{q*} = 1 + \frac{1}{p*} -
    \frac{1}{q},\\ &\norm{ h a }_{p*} \leq \norm{h}_{r_2}
    \norm{a}_{q*}, && \frac{1}{r_2} = \frac{1}{2} + \frac{1}{p*} -
    \frac{1}{q},\\ &\norm{ \db h a }_p \leq \norm{\db h}_{r_1}
    \norm{a}_{q*},\\ &\norm{ h \db a }_p \leq \norm{h}_{r_2} \norm{\db
      a}_q.
  \end{align*}
  Lemmas \ref{bumpNorm} and \ref{hNorm} then give
  \begin{align*}
    &\norm{ \psi_\tau a}_p \leq C \tau^{-1-1/{p*}+1/q} \norm{a}_{q*},
    \\ &\norm{ h a }_{p*} \leq C \tau^{-1/{p*}+1/q} \norm{a}_{q*},
    \\ &\norm{ \db h a }_p \leq C \tau^{-1/{p*}+1/q} \norm{a}_{q*}
    \\ &\norm{ h \db a }_p \leq C \tau^{-1/{p*}+1/q} \norm{\db a}_q,
  \end{align*}
  which implies the claim.
\end{proof}

\begin{proposition} \label{Sestimates}
  Let the usual assumptions hold and let $q_1,q_2 \in L^p(\Omega)$
  where $1 < p < 2$. Then we can extend $S_1$
  and $S_2$ from Definition~\ref{Sdef} to the following maps with
  corresponding norm estimates
  \begin{align*}
    S_j:L^\infty(X) \to L^{p*}(X), & \qquad \norm{S_j f}_{p*} \leq C
    \tau^{-1/2} \norm{f}_\infty,\\ S_j:L^\infty(X) \to L^\infty(X), &
    \qquad \norm{S_j f}_\infty \leq C \tau^{-\SinfinfExp}
    \norm{f}_\infty,
  \end{align*}
  where $C = C(p,\chi) \norm{q_j}_p$ and $0<\SinfinfExp<1/p$ with
  $\SinfinfExp=\SinfinfExp(p)$. If in addition $p > 4/3$ then we have
  the extension
  \[
  S_j:L^{p*}(X) \to L^{{p*}/2}(X), \qquad \norm{S_j f}_{{p*}/2} \leq C
  \tau^{-1/2} \norm{f}_{p*}.
  \]
\end{proposition}
\begin{proof}
  We shall prove the claim for $j=1$. The other case follows
  similarly. Using the notation from Lemma \ref{fundLpEstim} we can
  write $-4 S_1 f = \conjCa( \chi \Cab(e^{i\tau(\Phi+\overline\Phi)}
  q_1 f ) )$. The lemma combined with Lemma \ref{cauchyNorm} gives us
  \[
  \norm{S_1 f}_{q*} \leq C \tau^{-1/2} \norm{\chi
    \Cab(e^{i\tau(\Phi+\overline\Phi)} q_1 f )}_{W^{1,q}} \leq C
  \tau^{-1/2} \norm{q_1 f}_q
  \]
  whenever $2<q*<\infty$ and $1/q = 1/2 + 1/{q*}$. For the first
  estimate choose $q=p$, ${q*}={p*}$, and for the third one
  $1/q=1/p+1/{p*}$, ${q*}={p*}/2$. H\"older's inequality implies the
  rest.

  The second claim follows by interpolation. Let $2<Q<\infty$ and
  $1<q<p$. If $q_1 \in L^Q(X)$ then by Lemma \ref{fundLinftyEstim}
  \[
  \norm{S_1 f}_\infty \leq C \tau^{-1/Q} \norm{\chi
    \Cab(e^{i\tau(\Phi+\overline\Phi)} q_1 f )}_{W^{1,Q}}
  \]
  and Lemma \ref{cauchyNorm} gives the bound $C \tau^{-1/Q}
  \norm{q_1}_Q \norm{f}_\infty$· The latter lemma and Sobolev
  embedding imply that $\Cab : L^q \to L^{q*}$ and $\Ca : L^{q*} \to
  L^\infty$. Hence $\norm{ S_1 f }_\infty
  \leq C \norm{q_1}_q \norm{f}_\infty$.  Since $q<p<Q$ and $1/Q>0$
  interpolation gives us the second estimate with some
  $\SinfinfExp>0$.
\end{proof}

\begin{lemma} \label{varphiEstimates}
  Let the usual assumptions hold and $q_j \in L^p(\Omega)$ with
  $1<p<2$. Then $\varphi_j$, the function of $z\in X$ given by
  Definition \ref{varphiDef}, is in $L^\infty(X)$ with norm
  \[
  \norm{\varphi_j}_\infty \leq C \tau^{-\SinfinfExp}
  \]
  for $\SinfinfExp>0$ as in Proposition \ref{Sestimates}, and is in
  $L^{p*}(X)$ satisfying
  \[
  \norm{\varphi_j}_{p*} \leq C \tau^{-1/2},
  \]
where $C=C(p,\chi)(\norm{q_j}_p+\abs{\beta_j(z_0)})$.
\end{lemma}
\begin{proof}
  Note that $\varphi_1 = \frac{1}{4}\beta_1(z_0) \Ca\big(
  e^{-i\tau(\Phi+\overline\Phi)} \chi\big) +
  S_1(e^{-i\tau(\Phi+\overline\Phi)})$ and use Proposition
  \ref{Sestimates} and lemmas \ref{fundLinftyEstim} and
  \ref{fundLpEstim}.
\end{proof}

\bigskip
\begin{proof}[Proof of Theorem \ref{CGOexistence}]
  By Proposition \ref{Sestimates} and Lemma \ref{varphiEstimates} we
  have
  \[
  \norm{S_j f}_\infty \leq C(p,\chi) \norm{q_j}_p \tau^{-\alpha}
  \norm{f}_\infty, \quad \norm{\varphi_j}_\infty \leq
  C(p,\chi)(\norm{q_j}_p + \abs{\beta_j(z_0)}) \tau^{-\alpha}
  \]
  where these norms are over the variable $z\in X$. Hence we get
  $\norm{F_{j,m}}_\infty \leq C^m \tau^{-m\SinfinfExp}$ for some
  $C=C(p,\chi,q_j,\sup_{z_0}\abs{\beta_j(z_0)})$ when $m\geq0$. If
  $\tau > C^{1/\SinfinfExp}$ then the series for $f_j$ converges in
  $L^\infty(X)$.

  The following observations, which are each easy to check in
  $\mathscr D'(X)$, imply that $D_j f_j = q_j f_j$ in $\Omega$. Note
  that $\beta_j$ are functions of the parameter $z_0$ but constant in
  the variable $z$. Recall the definitions \ref{dDef}, \ref{Sdef} and
  \ref{varphiDef} of $D_j$, $S_j$ and $\varphi_j$. Then
  \begin{itemize}
    \item $f_j = e^{-i\tau(\Phi+\overline\Phi)} + \varphi_j + S_j(f_j
      - e^{-i\tau(\Phi+\overline\Phi)})$,
    \item $D_j(e^{-i\tau(\Phi+\overline\Phi)})=0$,
    \item $D_j S_j f = q_j f$ in $\Omega$,
    \item $D_j \varphi_j = q_j e^{-i\tau(\Phi+\overline\Phi)}$ in
      $\Omega$.
  \end{itemize}
  Lemma \ref{conjPDE} shows that we indeed get solutions to
  $(\Delta+q_j)u_j=0$.

  Recall that $1/p=1/2+1/{p*}$. Then by Lemma \ref{cauchyNorm} we have
  $\Ca, \Cab: L^{p*} \to W^{1,{p*}}$ which embeds to $W^{1,2}$ locally
  since ${p*}>2$. Moreover by Sovolev embedding we have $\Ca, \Cab :
  L^p \to L^{p*}$. Hence by the first item above $f_j \in W^{1,2}(X)$.
\end{proof}

\section{Proof of the main result} \label{alessandriniSection}

We will prove Theorem \ref{uniquenessThm} in this section. The proof
will be split into several lemmas. By Alessandrini's identity
\[
\int_\Omega (q_1-q_2) u_1 u_2 dx = 0
\]
for solutions $u_j$ to $(\Delta+q_j)u_j=0$ in $\Omega$. For any
parameter $z_0\in\C$ let $u_j$ be the complex geometrical optics
solution given by Theorem \ref{CGOexistence}. Recall that they are
defined in $X \supset \Omega$ but are solutions only in $\Omega$. Then
the product $u_1u_2$ will be a series of terms, and these will have to
be estimated carefully. Lemmas \ref{alessandriniTermByTerm} --
\ref{0orderTerms} deal with this. The main proof follows.

\begin{lemma} \label{alessandriniTermByTerm}
  Let the usual assumptions hold and $1<p<2$ with $q_1,q_2 \in
  L^p(\Omega)$. Let $f_1, f_2 \in L^\infty(X)$ be as in Theorem
  \ref{CGOexistence} and set $u_1 = e^{i\tau\Phi}f_1$ and $u_2 =
  e^{i\tau\overline\Phi}f_2$. Then
  \[
  \frac{2\tau}{\pi} \int (q_1-q_2) u_1 u_2 dx = \sum_{k+l=0}^\infty
  \frac{2\tau}{\pi} \int (q_1-q_2) e^{i\tau(\Phi+\overline\Phi)}
  F_{1,k} F_{2,l} dx
  \]
  where $k,l\geq0$ and the sum converges in the $L^\infty(X)$-norm
  with respect to $z_0$.
\end{lemma}
\begin{proof}
  We can estimate
  \[
  \abs{\int \sum_{k+l=N}^\infty (q_1-q_2)
    e^{i\tau(\Phi+\overline\Phi)} F_{1,k} F_{2,l} dx } \leq \int
  \sum_{k+l=N}^\infty \abs{q_1-q_2} \abs{F_{1,k}} \abs{F_{2,l}} dx
  \]
  and use the $z_0$-independent estimates for $F_{j,m}$ from Theorem
  \ref{CGOexistence} to see that the remainder tends to zero as
  $N\to\infty$. Hence the sum can be taken out of the integral and the
  claim follows.
\end{proof}

\begin{lemma} \label{highOrderTerms}
  Let the usual assumptions hold and $q_1, q_2 \in L^p(\Omega)$ with
  $4/3<p<2$. For $j\in\{1,2\}$ and $m\in\N$ take $F_{j,m}$ as in
  Definition \ref{fDef}.

  Then if $q_1-q_2\in L^2(\Omega)$ we have
  \[
  \abs{ \tau \int (q_1-q_2) e^{i\tau(\Phi+\overline\Phi)} F_{1,k}
    F_{2,l} dx } \leq C^{k+l} \tau^{-(k+l-2)\SinfinfExp}.
  \]
  when $k+l\geq3$. Here
  $C=C(p,\chi,\norm{q_1}_p,\norm{q_2}_p,\norm{q_1-q_2}_2)
  (1+\abs{\beta_1(z_0)}+\abs{\beta_2(z_0)})$ and $\SinfinfExp>0$ is as
  in Proposition \ref{Sestimates}.
\end{lemma}

\begin{proof}
  We may assume that $k\geq l$. By H\"older's inequality the integral
  can be estimated with
  \[
  C \tau \norm{q_1-q_2}_2 \norm{F_{1,k}}_{{p*}/2} \norm{F_{2,l}}_\infty
  \]
  because $p>4/3$ imples ${p*}>4$ for $1/p=1/2+1/{p*}$ and then $1/2 +
  2/{p*} \leq 1$. Proposition \ref{Sestimates} and Lemma
  \ref{varphiEstimates} imply the following estimates
  \begin{align*}
    &\norm{F_{j,0}}_\infty = 1,\quad \norm{F_{j,1}}_\infty \leq C
    \tau^{-\SinfinfExp},\quad \norm{F_{j,1}}_{p*} \leq C \tau^{-1/2},
    \\ &\norm{F_{j,m+1}}_\infty \leq C \tau^{-\SinfinfExp}
    \norm{F_{j,m}}_\infty, \\ &\norm{F_{j,m+1}}_{p*} \leq C
    \tau^{-1/2} \norm{F_{j,m}}_\infty, \\ &\norm{F_{j,m+1}}_{{p*}/2}
    \leq C \tau^{-1/2} \norm{F_{j,m}}_{p*}
  \end{align*}
  for $j\in\{1,2\}$ and $m=1,2,\ldots$. These imply
  \[
  \norm{F_{1,k}}_{{p*}/2} \leq C^k \tau^{-1-(k-2)\SinfinfExp},\qquad
  \norm{F_{2,l}}_\infty \leq C^l \tau^{-l\SinfinfExp}
  \]
  for $k\geq2$, $l\geq0$. The claim is direct consequence.
\end{proof}

From Lemma~\ref{highOrderTerms}, we can see that the higher order
terms decay in $\tau$ whenever $k+l\ge 3$. A more refined estimate
shows that the term of $k+l=2$ also decays.
\begin{lemma} \label{2orderTerms}
  Let the usual assumption hold and $q_1,q_2\in L^p(\Omega)$ with
  $4/3<p<2$. For $j\in\{1,2\}$ and $m\in\N$ let $F_{j,m}$ be as in
  Definition \ref{fDef}. Assume that $q_1-q_2\in L^2(\Omega)$. For
  $k+l=2$, we have
  \[
  \abs{ \tau \int (q_1-q_2) e^{i\tau(\Phi+\overline\Phi)} F_{1,k}
    F_{2,l} dx } \leq C \tau^{1/p-3/4},
  \]
  where the constant $C$ is of the form
  $C(p,\chi,\norm{q_1}_p,\norm{q_2}_p,\norm{q_1-q_2}_2)
  (1+\abs{\beta_1(z_0)}+\abs{\beta_2(z_0)})$.
\end{lemma}
\begin{proof}
  We can assume that $k\geq l$. There are two cases: $k=2, l=0$ and
  $k=l=1$. Start with the first one. The integral with $F_{1,2}
  F_{2,0}$ is
  \[
  -\frac{1}{4} \tau \int (q_1-q_2) \Ca\big(
  e^{-i\tau(\Phi+\overline\Phi)} \chi \Cab(
  e^{i\tau(\Phi+\overline\Phi)} q_1 \varphi_1 ) \big) dx
  \]
  by Definition \ref{Sdef}. We have $q_1-q_2\in L^2(\Omega)$ and hence
  we should take the $L^{r*}(X)$-norm of the remaining factor for any
  ${r*}\geq2$.

  We note that $\varphi_1 \in L^{p*}(X)$ and $q_1\in L^p(X)$. Hence
  their product is in $L^q(X)$ with $1/q=1/p+1/{p*}=2/p-1/2$. Choose
  $1/{r*}=1/p-1/4$. Then $2<{r*}<\infty$ and $1/2+1/{r*}\geq1/q>1/2$
  since $4/3<p<2$. Hence by Lemma \ref{fundLpEstim}
  \[
  \norm{ \Ca\big( e^{-i\tau(\Phi+\overline\Phi)} \chi \Cab(
    e^{i\tau(\Phi+\overline\Phi)} q_1 \varphi_1) \big) }_{r*} \leq C
  \tau^{1/q-1-1/{r*}} \norm{q_1}_p \norm{\varphi_1}_{p*}
  \]
  and the exponential is $1/q-1-1/{r*} = 1/p-5/4$. Recall that
  $\norm{\varphi_1}_{p*} \leq C\tau^{-1/2}$ by Lemma
  \ref{varphiEstimates}. The claim for $k=2$, $l=0$ follows.

  In the case $k=l=1$ note that $\beta_1(z_0)-\Cab q_1 \in W^{1,p}(X)$
  by Lemma \ref{cauchyNorm}. Note that $4/3<p<2$ implies
  $1/2+1/4\geq1/p>1/2$. Then by Lemma \ref{fundLpEstim}
  \[
  \norm{\varphi_1}_4 \leq C \tau^{1/p-1-1/4} \norm{\beta_1(z_0)-\Cab
    q_1}_{W^{1,p}} \leq C \tau^{1/p-5/4} (\abs{\beta_1(z_0)} +
  \norm{q_1}_p).
  \]
  When $k=l=1$, the absolute value of the integral in the lemma
  statement becomes
  \[
  \abs{ \tau \int (q_1-q_2) e^{i\tau(\Phi+\overline\Phi)} \varphi_1
    \varphi_2 dx } \leq \norm{q_1-q_2}_2 \tau
  \norm{\varphi_1}_4 \norm{\varphi_2}_4
  \]
  by H\"older's inequality. The claim follows since $\tau^{2/p-6/4} <
  \tau^{1/p-3/4}$ when $\tau>1$ and $p>4/3$.
\end{proof}

\bigskip
We recall the method of stationary phase and its convergence in the
$L^2$-sense before proceeding to deal with terms of order one and zero
in the Alessandrini identity.

\begin{lemma} \label{statPhase}
  For $z_0\in\C$, $\Phi(z) = (z-z_0)^2$ and $\tau\in\R$ define the
  operator
  \[
  Ef(z_0) = \frac{2\tau}{\pi} \int e^{-i\tau(\Phi+\overline\Phi)} f(z)
  dm(z)
  \]
  for $f \in C^\infty_0(\C)$. Here $dm(z)$ is the two-dimensional
  Lebesgue measure in $\C$.

  Then $E$ can be extended to a unitary operator on $L^2(\C)$ such
  that
  \[
  \lim_{\tau\to\pm\infty} \norm{Ef-f}_2 = 0.
  \]
\end{lemma}  
\begin{proof}
  Consider the function $z \mapsto 2\tau \exp(-i(z^2+\overline
  z^2))/\pi$ defined on $\C\equiv\R^2$. Its Fourier transform is
  $\exp(i(\xi^2+\overline\xi^2)/(16\tau))$ by for example
  \cite{Blasten--Imanuvilov--Yamamoto}. We have $Ef =
  \frac{2\tau}{\pi} e^{-i(z^2+\overline z^2)} \ast f$ and hence
  $\mathscr F\left\{Ef\right\}(\xi) =
  e^{i\frac{\xi^2+\overline\xi^2}{16\tau}} \hat f(\xi)$. Parseval's
  theorem implies the unitary extension to $L^2(\C)$. When
  $\tau\to\pm\infty$ the exponential tends to $1$ pointwise. Dominated
  convergence and Parseval's theorem imply the second claim.
\end{proof}

The following way of dealing with the first order terms comes from
\cite{imanuvilovYamamotoLp, Blasten--Imanuvilov--Yamamoto}.

\begin{lemma} \label{1orderTerms}
  Let the usual assumptions hold and $q_1,q_2 \in L^p(\Omega)$ with
  $4/3<p<2$. For $j\in\{1,2\}$ and $m\in\N$, let $F_{j,m}$ be as in
  Definition \ref{fDef}. Moreover let $\beta_j \in L^\infty(X)$ with
  respect to the $z_0$-variable. Then
  \[
  \lim_{\tau\to\infty} \norm{ \frac{2\tau}{\pi} \int (q_1-q_2)
    e^{i\tau(\Phi+\overline\Phi)} F_{1,k} F_{2,l} dx}_2 \leq C \norm{
    \beta_2 - \Ca q_2 }_{p*} \norm{q_1-q_2}_p
  \]
  for $k=1$, $l=0$, where the $L^2(X)$-norm is taken over the variable
  $z_0$ and $C=C(p,\chi)$. A similar bound holds for $k=0$ and $l=1$.
\end{lemma}
\begin{proof}
  Recall that $\varphi_2 = \frac{1}{4} \Cab\big(
  e^{-i\tau(\Phi+\overline\Phi)} \chi (\beta_2(z_0) - \Ca q_2)\big)$
  and hence the integral becomes
  \[
  \frac{\tau}{2\pi} \int (q_1-q_2) \Cab\big(
  e^{-i\tau(\Phi+\overline\Phi)} \chi (\beta_2(z_0) - \Ca q_2) \big)
  dx
  \]
  when $k=1$, $l=0$. By Fubini's theorem this is equal to
  \[
  - \frac{\tau}{2\pi} \int e^{-i\tau(\Phi+\overline\Phi)} \chi
  (\beta_2(z_0) - \Ca q_2) \Cab(q_1-q_2) dx,
  \]
  and using the stationary phase operator of Lemma \ref{statPhase} it
  is equal to
  \[
  \frac{1}{4} E\big(\chi \Ca q_2 \Cab(q_1-q_2)\big)(z_0) - \frac{1}{4}
  \beta_2(z_0) E\big( \chi \Cab(q_1-q_2) \big)(z_0).
  \]

  We have $\Cab(q_1-q_2)\in L^{p*}(X)$ where ${p*}>4$ since $p>4/3$
  (e.g. Lemma \ref{cauchyNorm} and Sobolev embedding). Similarly $\chi
  \Ca q_2 \in L^{p*}(\C)$. Their product is in $L^2(\C)$ since $\chi$
  has compact support. Hence the operator $E$ is being applied to
  $L^2(\C)$-functions above. Since $z_0 \mapsto \beta_2(z_0)$ is
  uniformly bounded, the above converges to
  \[
  \frac{1}{4}\chi (\Ca q_2 - \beta_2) \Cab(q_1-q_2)
  \]
  in the $L^2(\C)$-norm with respect to $z_0$ as $\tau\to\infty$ by
  Lemma \ref{statPhase}. The claim follows from the norm estimates at
  the beginning of this paragraph.
\end{proof}

\begin{lemma} \label{0orderTerms}
  Let the usual assumptions hold and $q_1-q_2 \in L^2(\Omega)$. Then
  \[
  \lim_{\tau\to\infty} \norm{ \frac{2\tau}{\pi} \int (q_1-q_2)
    e^{i\tau(\Phi+\overline\Phi)} F_{1,0} F_{2,0} dx - (q_1-q_2) }_2 =
  0,
  \]
  where $F_{j,0}$ are as in Definition \ref{fDef} and the norm is
  taken with respect to the variable $z_0\in\C$.
\end{lemma}
\begin{proof}
  This follows directly from $F_{j,0} =
  \exp(-i\tau(\Phi+\overline\Phi))$ and the stationary phase Lemma
  \ref{statPhase}.
\end{proof}

\bigskip
We are ready to prove uniqueness for the inverse problem with
potential in $L^p$, $4/3<p<2$.

\begin{proof}[Proof of Theorem \ref{uniquenessThm}]
  In view of Green's identity and the symmetry of the DN map, we can
  see that the condition of identical Dirichlet-Neumann maps imply
  that
  \[
  \int_\Omega (q_1-q_2) u_1 u_2 dx = 0
  \]
  for any solution $u_j \in W^{1,2}(\Omega)$ to $(\Delta+q_j)u_j=0$ in
  $\Omega$. We also note that Theorem 3 in \cite{Serov--Paivarinta}
  implies that $q_1-q_2 \in H^t(\Omega)$ for any $t<3-4/p$. Hence
  $q_1-q_2 \in L^2(\Omega)$.

  Let $\varepsilon>0$ and take $\beta_j\in C^\infty_0(X)$ such that
  \begin{equation}\label{0411-1}
    \lVert \beta_1 - \Cab q_1 \rVert_{L^{p*}(X)} < \varepsilon, \quad
    \lVert \beta_2 - \Ca q_2 \rVert_{L^{p*}(X)} < \varepsilon
  \end{equation}
  which is possible since $\Cab q_1, \Ca q_2 \in L^{p*}$ by Sobolev
  embedding and Lemma \ref{cauchyNorm}. Let $z_0\in\C$ and from now
  $\beta_1$ and $\beta_2$ shall be evaluated at $z_0$ if not mentioned
  otherwise, and note that they are uniformly bounded. Then, given
  $\tau>1$ large enough let $u_1 = e^{i\tau\Phi}f_1$ and $u_2 =
  e^{i\tau\overline\Phi}f_2$ be the solutions in the variable $z$ with
  parameter $z_0$, given by Theorem \ref{CGOexistence}. They are in
  $W^{1,2}(X)$ and satisfy $(\Delta+q_j)u_j=0$ in $\Omega$. By Lemma
  \ref{alessandriniTermByTerm} we have
  \[
  \frac{2\tau}{\pi} \int (q_1-q_2) u_1 u_2 dx = \sum_{k+l=0}^\infty
  \frac{2\tau}{\pi} \int (q_1-q_2) e^{i\tau(\Phi+\overline\Phi)}
  F_{1,k} F_{2,l} dx.
  \]

  In view of  lemmas \ref{highOrderTerms} and \ref{2orderTerms}
  \begin{align*}
    &\sum_{k+l=2}^\infty \frac{2\tau}{\pi} \abs{\int (q_1-q_2)
      e^{i\tau(\Phi+\overline\Phi)} F_{1,k} F_{2,l} dx} \leq C
    \tau^{1/p - 3/4} + \sum_{k+l=3}^\infty C^{k+l}
    \tau^{-(k+l-2)\SinfinfExp} \\ &\qquad = C \left( \tau^{1/p - 3/4}
    + \sum_{N=3}^\infty (N+1) (C\tau)^{-(N-2)\SinfinfExp} \right)
  \end{align*}
  for any $z_0 \in \R^2$, and where $C = C_{p,q_1,q_2,\Omega} ( 1 +
  \norm{\beta_1}_\infty + \norm{\beta_2}_\infty)$. Recall that $p>4/3$
  so the first exponent is negative. Note that for $\tau$ sufficiently
  large, $(C\tau)^{-\SinfinfExp} < 1$ so the sum can be rewritten as
  \begin{align*}
    &\sum_{N=3}^\infty (N+1) ((C\tau)^{-\SinfinfExp})^{N-2} =
    \sum_{N=1}^\infty (N+1) ((C\tau)^{-\SinfinfExp})^N + 2
    \sum_{N=1}^\infty ((C\tau)^{-\SinfinfExp})^N \\ &\qquad =
    \frac{1}{(1-(C\tau)^{-\SinfinfExp})^2} - 1 +
    \frac{2(C\tau)^{-\SinfinfExp}}{1-(C\tau)^{-\SinfinfExp}},
  \end{align*}
  which tends to zero as $\tau\to\infty$. Hence the sum of the terms
  with $k+l\geq2$ in the original sum tends to zero when $\beta_j$ are
  fixed.

  For the terms with $k+l\in\{0,1\}$ we will use lemmas
  \ref{1orderTerms} and \ref{0orderTerms}. By them
  \begin{align*}
    &\lim_{\tau\to\infty} \norm{\sum_{k+l=0}^1 \frac{2\tau}{\pi} \int
      (q_1-q_2) e^{i\tau(\Phi+\overline\Phi)} F_{1,k} F_{2,l} dx -
      (q_1-q_2)}_{L^2(X)} \\ &\qquad \leq C( \lVert \beta_1 - \Cab q_1
    \rVert_{L^{p*}(X)} + \lVert \beta_2 - \Ca q_2 \rVert_{L^{p*}(X)} )
    \leq 2C \varepsilon
  \end{align*}
  where the $L^2(X)$-norm is taken with respect to $z_0$ and this time
  $C$ does not depend on $\beta_1$ or $\beta_2$. We can redo this
  whole argument for any $\varepsilon>0$, and thus by Alessandrini's
  identity
  \[
  \norm{q_1-q_2}_{L^2(\Omega)} \leq \lim_{\tau\to\infty} \norm{q_1-q_2
    - \sum_{k+l=0}^\infty \frac{2\tau}{\pi} \int (q_1-q_2)
    e^{i\tau(\Phi+\overline\Phi)} F_{1,k} F_{2,l} dx}_{L^2(\Omega)}
  \]
  the latter of which can be made as small as we please by choosing
  $\beta_1, \beta_2$. The claim follows.
\end{proof}

\section{Appendix 1: Cauchy operator and integration by parts} \label{cauchySection}

We define the two fundamental tools for solving the two-dimensional
inverse problem of the Schr\"odinger operator in this section: the
Cauchy operators and an integration by parts formula for the Cauchy
operator conjugated by an exponential. These were used by Bukhgeim
\cite{bukhgeim} for solving the problem.

\begin{definition} \label{cauchyDef}
  Let $u \in \mathscr{E}'(\R^2)$ be a compactly supported
  distribution. Then we define the \emph{Cauchy operators} by
  \[
  \Ca u = \frac{1}{\pi z} \ast u, \qquad \Cab u = \frac{1}{\pi
    \overline{z}} \ast u.
  \]
\end{definition}

\begin{remark} \label{cauchyWellDefined}
  The notations $\overline\partial^{-1}$ and $\partial^{-1}$ cause no
  problems because $1/(\pi z)$ and $1/(\pi \overline{z})$ are the
  fundamental solutions to the operators $\overline\partial =
  (\partial_1 + i\partial_2)/2$ and $\partial = (\partial_1 -
  i\partial_2)/2$.
\end{remark}

\begin{lemma}
  \label{intByPartsLemma}
  Let $\tau > 0$, $z_0 \in \C$ and $\Phi(z) = (z-z_0)^2$. Let $\psi
  \in C^\infty_0(\R^2)$ with $\psi \equiv 1$ in a neighbourhood of
  $0$, and write
  \[
  \psi_\tau(z) = \psi( \tau^{1/2}(z-z_0) ), \qquad h(z) = \frac{1 -
    \psi_\tau(z) }{\overline{z}-\overline{z_0}}.
  \]
Then for $a \in C^\infty_0(\R^2)$ we have the integration by parts formula
  \begin{align*}
    &\Ca(e^{-i\tau(\Phi + \overline{\Phi})}a) = \Ca\big(e^{-i\tau(\Phi
      + \overline{\Phi})} \psi_\tau a\big) \\ &\qquad\quad -
    \frac{1}{2i\tau} \big( e^{-i\tau(\Phi + \overline{\Phi})} h a -
    \Ca(e^{-i\tau(\Phi + \overline{\Phi})} \db h a) - \Ca(
    e^{-i\tau(\Phi + \overline{\Phi})} h \db a) \big).
  \end{align*}
  If we had set $h(z) = (1-\psi_\tau(z))/(z-z_0)$ instead then
  \begin{align*}
    &\Cab(e^{-i\tau(\Phi + \overline{\Phi})} a) =
    \Cab\big(e^{-i\tau(\Phi + \overline{\Phi})} \psi a\big)
    \\ &\qquad\quad - \frac{1}{2i\tau}\big( e^{-i\tau(\Phi +
      \overline{\Phi})} h a - \Cab(e^{-i\tau(\Phi + \overline{\Phi})}
    \d h a) - \Cab(e^{-i\tau(\Phi + \overline{\Phi})} h \d a)\big).
  \end{align*}
\end{lemma}
\begin{proof}
  The proof follows by differentiating $e^{-i\tau(\Phi+\overline\Phi)}
  h a$ and noting that by Remark \ref{cauchyWellDefined} the operators
  $\Ca \db$ and $\Cab \d$ are the identity on compactly supported
  distributions.
\end{proof}

\begin{lemma} \label{cauchyNorm}
  Let $X \subset \R^2$ be a bounded domain and $1<p<\infty$. Then the
  Cauchy operators $\Ca$ and $\Cab$ are bounded $L^p(X) \to W^{1,p}(X)$.
\end{lemma}

\begin{proof}
  If $f\in L^p(X)$ we extend it by zero to $\R^2\setminus X$ to create
  a compactly supported distribution and thus $\Ca f$ is well defined
  by Definition \ref{cauchyDef}. The convolution kernel $1/(\pi z)$ is
  locally integrable, so by Young's inequality
  \[
  \norm{ \Ca f }_{L^p(X)} \leq C \norm{f}_{L^p(X)},
  \]
  because in essence $\Ca f$ has the same values in $X$ as the
  convolution of $f$ with the kernel $\chi_{X-X}(z)/(\pi z)$, where
  $X-X = \{ z \in \R^2 \mid z = z_1-z_2, z_j \in \R^2\}$.

  For the derivatives note that by Remark \ref{cauchyWellDefined} we
  have $\db \Ca f = f$. On the other hand $\d \Ca f = \Pi f$ which is
  the Beurling transform, and hence bounded $L^p(X) \to L^p(X)$. For
  reference see for example Section 4.5.2 in
  \cite{Astala--Iwaniec--Martin} or \cite{Vekua} for a more classical
  approach.
\end{proof}

\section{Appendix 2: Cut-off function estimates} \label{cutoffs}

This section contains all the technical cut-off function construction
and norm estimates used in the paper.

\begin{lemma} \label{bumpNorm}
  Let $\psi \in C^\infty_0(\R^2)$. For $z_0\in\R^2$ and $\tau>0$ write
  $\psi_\tau(z) = \psi( \tau^{1/2}(z-z_0) )$. Then, given any vector
  $v\in\C^2$, we have
  \[
  \norm{\psi_\tau}_{L^p(\R^2)} = \norm{\psi}_{L^p(\R^2)} \tau^{-1/p},
  \qquad \norm{ v\cdot\nabla \psi_\tau }_{L^p(\R^2)} = \norm{
    v\cdot\nabla \psi }_{L^p(\R^2)} \tau^{1/2-1/p}
  \]
  for $1\leq p \leq \infty$.
\end{lemma}

\begin{proof}
  This follows directly from the scaling properties and translation
  invariance of $L^p$-norms in $\R^2$.
\end{proof}

\begin{lemma} \label{1/zNorm}
  Let $\tau>0$ and set $\R^2_\tau = \R^2 \setminus
  B(0,\tau^{-1/2})$. Then
  \[
  \norm{ z^{-a} }_{L^p(\R^2_\tau)} = \left( \frac{2\pi}{ap-2}
  \right)^{1/p} \tau^{a/2-1/p}
  \]
  for $a>0$ and $2/a<p\leq\infty$.
\end{lemma}

\begin{proof}
  This is a direct computation using the polar coordinates integral
  transform $\int_{\R^2_\tau} \ldots dz = \int_{\tau^{-1/2}}^\infty
  \int_{\mathbb S^1} \ldots d\sigma(\theta) r dr$, with $z = r\theta$.
\end{proof}

\begin{lemma} \label{hNorm}
  Let $\psi \in C^\infty_0(\R^2)$ be a test function supported in
  $B(0,2)$ with $0\leq\psi\leq1$ and $\psi \equiv 1$ in $B(0,1)$.  For
  $\tau>0$ and $z_0\in\R^2$ write $\psi_\tau(z) = \psi(
  \tau^{1/2}(z-z_0) )$. Let $h(z) = (1 - \psi_\tau(z)) / (\overline{z}
  - \overline{z_0})$. Then
  \[
  \norm{h}_{L^p(\R^2)} \leq C_p \tau^{1/2-1/p}
  \]
  for $C_p<\infty$ when $2<p\leq\infty$ and for any complex vector
  $v\in\C^2$ we have
  \[
  \norm{v\cdot\nabla h}_{L^p(\R^2)}\leq C_{\psi,p,v} \tau^{1-1/p}
  \]
  for $C_{\psi,p,v}<\infty$ when $1\leq p \leq \infty$. The same
  conclusions hold if we had defined $h$ by dividing $1-\psi_\tau$ by
  $z-z_0$ instead of its complex conjugate.
\end{lemma}

\begin{proof}
  For the first claim note that $\abs{h(z)} \leq \abs{z-z_0}^{-1}$ and
  $\supp h \subset \R^2_\tau + z_0 = \R^2 \setminus
  B(z_0,\tau^{-1/2})$. Hence $\norm{h}_{L^p(\R^2)} \leq \norm{ z^{-1}
  }_{L^p(\R^2_\tau)}$ and Lemma \ref{1/zNorm} takes care of the first
  estimate.

  For the second estimate
  \[
  v\cdot\nabla h(z) = \frac{v\cdot\nabla
    \psi_\tau(z)}{\overline{z}-\overline{z_0}} -
  \frac{1-\psi_\tau(z)}{(\overline{z}-\overline{z_0})^2}.
  \]
  The $L^p$-norm of the first term is bounded by $\norm{ v\cdot\nabla
    \psi_\tau }_{L^p} \norm{z^{-1}}_{L^\infty(\R^2_\tau)}$ which is at
  most $C_{\psi,p,v} \tau^{1-1/p}$ according to lemmas \ref{bumpNorm}
  and \ref{1/zNorm}. The second term is supported in $\R^2\setminus
  B(z_0,\tau^{-1/2})$ and bounded pointwise by
  $\abs{z-z_0}^{-2}$. Hence, as in the first paragraph, it has the
  required bound.
\end{proof}

\section*{Acknowledgements}
Wang was supported in part by MOST 105-2115-M-002-014-MY3.\\
Tzou was supported in part by ARC FT130101346 and VR 2012-3782.

\bibliographystyle{plain}
\bibliography{ownBib}

\end{document}